\documentclass[12pt]{article}
\usepackage[english]{babel}
\usepackage[cp1251]{inputenc}

\usepackage{amsthm}
\usepackage{hhline}
\usepackage{amscd}

\usepackage{amsfonts, amscd, amsmath, amssymb}
\usepackage[dvips]{graphicx}
\usepackage[cp1251]{inputenc}
\usepackage[T2A]{fontenc}
\usepackage{mathrsfs}
\usepackage{amsthm}
\usepackage{amssymb}
\usepackage{amsmath,amsthm,amsfonts}
\usepackage{hhline}
\usepackage{textcomp}
\usepackage{amsmath,amscd}
\usepackage[all,cmtip]{xy}
\usepackage{tipa}
\usepackage{mathrsfs}
\usepackage{textcomp}
\usepackage{mathabx}

\usepackage{amsthm}
\usepackage{amssymb}
\usepackage{amsmath,amsthm,amsfonts}
\usepackage{hhline}
\usepackage{textcomp}
\newtheorem{theorem}{Theorem}[section]
\newtheorem{lemma}{Lemma}[section]
\newtheorem{rem}{Remark}[section]

\newtheorem{prop}{Proposition}[section]

\theoremstyle{definition}

\begin{document} 
\title{ A note on the Gaussian curvature of harmonic surfaces  }
\author{Kaveh Eftekharinasab\footnote{Institute of mathematics of Ukrainian Academy of Sciences. E-mail: kaveh@imath.kiev.ua}}
\date{}
\maketitle
\begin{abstract}
It was proved that the fundamental group of the space of harmonic polynomials of degree $n(n \geq 2)$ which their graphs have the same
Gaussian curvature is not trivial. Furthermore, we give an example of topologically nonequivalent conjugate
harmonic functions which their graphs have the same Gaussian curvature.
\end{abstract}

\section{Topological equivalency of harmonic functions }
Let $w = u(x,y)$ be  (at least) a $C^2$ harmonic function of real variables $x$ and $y$ defined on a region  $\Omega$ of $\mathbb{R}^2$.
 It is well known that its Gaussian curvature, denoted by $k(x,y)$, can be given by: $$k(x,y)= \dfrac{(u_{xx}u_{yy}-u_{xy}^2) }{ (1 + u_x^2 + u_y^2)}.$$ 
 Set $z = x + iy$ and rewrite $w = u(x,y)$ as follows:
\begin{equation*}
 u(x,y)= \operatorname{Re} f(z) = \dfrac{1}{2}(f(z) + \overline{f(z)}).
\end{equation*}
The partial differential operators of the holomorphic function $f(z)$ of the complex variable $z$ are given by:
\begin{equation*}
 \dfrac{\partial}{\partial z}=\dfrac{1}{2}\left(\dfrac{\partial}{\partial x} - \operatorname{i}\dfrac{\partial}{\partial y}\right),\quad
\dfrac{\partial}{\partial \overline{z}}=\dfrac{1}{2}\left(\dfrac{\partial}{\partial x} - \operatorname{i}\dfrac{\partial}{\partial y}\right),
\end{equation*}
These definitions are forced upon us by the logical requirement that 
\begin{equation*}
 \dfrac{\partial f}{\partial z} = f^{'}(z),\quad 
\dfrac{\partial f}{\partial \overline{z}} = 0,\quad
\dfrac{\partial \overline{f}}{\partial z} = 0,\quad 
\dfrac{\partial \overline{f}}{\partial \overline{z}} = \overline{f^{'}}(z).
\end{equation*}
It is easily seen that the Gaussian curvature of the graph obtained from  $w = u(x,y)$ can be
write as:
\begin{equation}\label{ga}
 K(x,y) = \dfrac{|f^{''}(z)|^2}{-(1 + |f^{'}(z)|^2)^2}.
\end{equation}
We recall that a critical point $p$ of a smooth function $w = g(x_1,x_2)$ is called non-degenerate
if the determinant of the Hessian matrix
\begin{equation*}
H(g) = \left( \dfrac{\partial^{2}g(x_1,x_2)}{\partial x_i \partial x_j}\right) 
\end{equation*}
dose not vanish at $p$. Otherwise, it is called degenerate critical point. 
\begin{lemma}\label{ka}
 The Gaussian curvature of the graph of the harmonic function $ w = u(x,y)$,
 defined on a region $\Omega$ is nonpositive, and vanishes only in isolated points which are
degenerate critical points of $w$. 
\end{lemma}
\begin{proof}
 Let $f(x+\operatorname{i}y)$ be a holomorphic function, and $w = u(x,y)$ its real part. The Gaussian curvature of the 
graph of $w=u(x,y)$ vanishes when $f''(x+\operatorname{i}y)= 0$. Since $f(x,y)$ is holomorphic, it follows that 
$f''(x+\operatorname{i}y)$ is holomorphic therefore, its zeros are isolated. Moreover, critical points of $f(x+\operatorname{i}y)$ 
and $w=u(x,y)$ coincide. The zeros of $f''(x+\operatorname{i}y)=0$ are critical points of $f(x+\operatorname{i}y)$ and 
hence critical points of $w=u(x,y)$. From the equation 
\begin{equation*}
 K(x,y) = \dfrac{u_{xx} u_{yy} - u^2_{xy}}{(1 + u^2_x + u^2_y)^2} = 0,
\end{equation*}
it follows that,
$u_{xx} u_{yy} - u^2_{xy} = 0$ if and only if critical points are degenerate.
\end{proof}
\begin{rem}
 Suppose $w=u(x,y)$ is a harmonic function of degree $n(n > 2)$. Then it is obvious that the cardinality
of the set $\Gamma$ of points at which the Gaussian curvature of the graph of $w=u(x,y)$ are zero, is not greater than
$n-2$. By Gauss-Lucas theorem (cf.~\cite{mar}) points of $\Gamma$ lie within the convex hull $\Delta$ containing
the zeros of holomorphic function $f(x+\operatorname{i}y)=u(x,y)+\operatorname{i} v(x,y)$. If the multiplicity of zeros of
 $f(x+\operatorname{i}y)$ is not greater than two, then the points of $\Gamma$ lie inside the  convex hall
 $\delta\subset\Delta$ containing the zeros of
$f'(x+\operatorname{i}y)$.
\end{rem}
\begin{rem}\label{eq}
 It is immediate from the Equation~\eqref{ga} that if $f= u +\operatorname{i}v$ is a holomorphic function,
then the graphs of $u$,$v$ and $f$ have the same Gaussian curvature, see~\cite{j}.   
 
\end{rem}
Recall that two smooth functions $w,v :\Omega \subset \mathbb{R}^2 \rightarrow \mathbb{R}$
are said to be  topologically equivalence if there exist homomorphisms $\phi : \Omega \rightarrow \Omega$ and
$\psi : \mathbb{R} \rightarrow \mathbb{R} $ such that $\psi \circ w = v\circ \phi$.

Let $f$ be a smooth function on $\Omega$, and $p$ a non-degenerate critical point of $f$. A component
of the level line $f^{-1}(p)$ is called fiber. Two functions $f$ and $g$ are called fiber equivalent 
at a non-degenerate critical point $p$ if there exists a homeomorphism which maps corresponding fibers to
each other, local extrema  to local extrema and saddle points to saddle points. Clearly, if $f$ and $g$ are topological
equivalent they are fiber equivalent at non-degenerate critical points. 
\begin{prop}
 There exist topological nonequivalent conjugate harmonic functions which their graphs have the same Gaussian curvature.
\end{prop}
\begin{proof}
  Consider conjugate harmonic functions $u(x,y)= x^3-3xy^2-3x$ and $v(x,y)=-y^3+3x^2y-3y,$.
 Each of these functions have
two non-degenerate critical points $(\pm 1,0)$. But for the function $u(x,y)$ these critical points lie
in different level lines, and for the function $v(x,y)$ critical points lie on one level line. Consequently, they
are not topologically equivalent, but their Gaussian curvature is the same because they are conjugate harmonic
functions.  
\end{proof}
\begin{rem}
 The graphs of functions $u(x,y)$ and $v(x,y)$ are not isometric, they have different first quadratic forms.
\end{rem}

\section{Gaussian Curvature of harmonic polynomials} 

We denote the Gaussian curvature of the graph a complex polynomial $P$ by $G_P$:
\begin{equation*}
 G_P \coloneq \dfrac{|P^{''}(z)|^2}{-(1 + |P^{'}(z)|^2)^2}.
\end{equation*}
\begin{lemma}
 The Gaussian curvature of the graph of harmonic polynomials $u(x,y)$ and $v(x,y)$ of different degrees are always different. 
\end{lemma}
\begin{proof}
 Let $Q(z)$ and $P(z)$ be complex polynomials of degree $n$ and $m$ resp., and let $u(x,y)$ and $v(x,y)$
be their real parts, resp. By assumption $m \neq n$. Functions $\lvert P(z) \rvert$ and $\lvert Q(z)\rvert$ 
attend to $\lvert z^n \rvert$ and $\rvert z^m \lvert$, resp., when $z$ goes to the infinity. Employing this fact for
\begin{equation}
 K(x,y) = \dfrac{|f^{''}(z)|^2}{-(1 + |f^{'}(z)|^2)^2},
\end{equation}
 we obtain that 
\begin{equation*}
 \dfrac{|P^{''}(z)|^2}{-(1 + |P^{'}(z)|^2)^2}\neq \dfrac{|Q^{''}(z)|^2}{-(1 + |Q^{'}(z)|^2)^2}.
\end{equation*}
Thus, the Gaussian curvature of polynomials $u(x,y)$ and $v(x,y)$ are different
\end{proof}
\begin{theorem}\label{hj}
 Suppose $P(z)$ and $Q(x)$ are complex polynomials. Then $G_P$ and 
$G_Q$ coincide if and only if $P(z)= \alpha Q(z)+ \beta$,
 where $\beta$ is a complex constant and $\alpha$ is a complex number such that$\rvert \alpha \lvert=1$.
\end{theorem}
\begin{proof}
 \textbf{Sufficiency:} If $P(z)= \alpha Q(z)+ \beta$, then from the Formula~\eqref{ga} it follows 
the conclusion of the theorem.\newline
\textbf{Necessity:}
We only need to show that
\begin{equation} 
\dfrac{|P^{''}(z)|}{1 + |P^{'}(z)|^2} \quad \emph{and} \quad \dfrac{|Q^{''}(z)|}{1 + |Q^{'}(z)|^2}
\end{equation}
are equal. Put $P'(z) \coloneq n(z)$, obviously $n(z)$ is polynomial. Assume $\gamma(t)$ is a path lies in a region where
$n(z)$ is defined, with the initial point $z_0$ and the end point $z$. Consider the  following  integrals

\begin{align*}
 \int_{\gamma(t)}\dfrac{|n^{'}(z)||d(z)|}{1 + |n(z)|^2} 
&= \int_{\gamma(t)}\dfrac{d(|n(z)|)}{1 + |n(z)|^2} \\
&= \arctan|n(z)| + const.
\end{align*}
Similarly, set $Q^{'}(z) \coloneq m(z)$ and perform the same arguments, we obtain 
\begin{align*}
 \int_{\gamma(t)}\dfrac{|m^{'}(z)||d(z)|}{1 + |m(z)|^2} 
&= \int_{\gamma(t)}\dfrac{d(|m(z)|)}{1 + |m(z)|^2} \\
&=\arctan|m(z)| + const.
\end{align*}
Thus
$\arctan|n(z)| = \arctan|m(z)| + \beta,$
where $\beta$ is constant, and therefore $$|n(z)| = |m(z)| + \beta.$$ 
Now we can easily show that $P(z)= \alpha Q(z)+ \beta$, where $\rvert \alpha \lvert = 1$.
\end{proof}
\begin{rem}
 Theorem~\ref{hj} is not valid for any holomorphic functions, see examples in~\cite{j}.    
\end{rem}
Let $w=P(x,y)$ and $v=Q(x,y)$ be conjugate harmonic polynomials, if the parameter $t$ varies in the unit circle
in the complex plane. Then 1-parameter family of polynomials $$\cos(t)u(x,y)-\sin(t)v(x,y)$$ forms a loop
in the space of harmonic polynomials which their graphs have the same Gaussian curvature. It is easily seen
 that the functions $\pm P(x,y)$ and $\pm Q(x,y)$ lie in the loop.
Hence, from Theorem~\ref{hj} and Lemma~\ref{ka} follows the next fact:
\begin{prop}
 The fundamental group of the space of harmonic polynomials of degree $n(n \geq 2)$ which their graphs have the same
Gaussian curvature is not trivial.
\end{prop}

\end{document}